\documentclass[12pt,fleqn,a4paper]{amsart}

\usepackage[utf8]{inputenc}
\usepackage[in]{fullpage}
\usepackage[colorlinks,breaklinks,linkcolor=blue,anchorcolor=blue,citecolor=blue]{hyperref}
\usepackage{amsfonts,amssymb,amsthm}
\usepackage[alphabetic,nobysame]{amsrefs}
\catcode`\@=11

\def\@settitle{%
  \baselineskip14\p@\relax
    {\Large\bfseries
  \@title}}

\def\@setauthors{%
  \begingroup
  \def\thanks{\protect\thanks@warning}%
  \trivlist
  \footnotesize \@topsep45\p@\relax
  \advance\@topsep by -\baselineskip
  \item\relax
  \author@andify\authors
  \def\\{\protect\linebreak}%
  {\sc\fontsize{12}{10}\selectfont\authors}%
  \ifx\@empty\contribs
  \else
    ,\penalty-3 \space \@setcontribs
    \@closetoccontribs
  \fi
  \endtrivlist
  \endgroup
}

\def\@secnumfont{\bfseries}%

\def\section{\@startsection{section}{1}%
  \z@{.7\linespacing\@plus\linespacing}{.5\linespacing}%
  {\normalfont\bf}}

\renewcommand{\BibLabel}{%
    \Hy@raisedlink{\hyper@anchorstart{cite.\CurrentBib}\hyper@anchorend}%
    [\thebib]\hfill%
}

\DefineSimpleKey{bib}{arxiv}

\BibSpec{article}{%
    +{}  {\PrintAuthors}                {author}
    +{,} { \textit}                     {title}
    +{.} { }                            {part}
    +{:} { \textit}                     {subtitle}
    +{,} { \PrintContributions}         {contribution}
    +{.} { \PrintPartials}              {partial}
    +{,} { }                            {journal}
    +{}  { \textbf}                     {volume}
    +{}  { \PrintDatePV}                {date}
    +{,} { \eprintpages}                {pages}
    +{,} { }                            {status}
    +{,} { \arxiv}           {arxiv}
    +{}  { \PrintTranslation}           {translation}
    +{;} { \PrintReprint}               {reprint}
    +{.} { }                            {note}
    +{.} {}                             {transition}
  }

\BibSpec{collection.article}{%
    +{}  {\PrintAuthors}                {author}
    +{,} { \textit}                     {title}
    +{.} { }                            {part}
    +{:} { \textit}                     {subtitle}
    +{,} { \PrintContributions}         {contribution}
    +{,} { \PrintConference}            {conference}
    +{}  {\PrintBook}                   {book}
    +{,} { }                            {booktitle}
    +{} { \PrintDate}                  {date}
    +{,} { }                        {pages}
    +{,} { }                            {status}
    +{,} { \PrintDOI}                   {doi}
    +{,} { available at \eprint}        {eprint}
    +{,} { \arxiv}           {arxiv}
    +{}  { \PrintTranslation}           {translation}
    +{;} { \PrintReprint}               {reprint}
    +{.} { }                            {note}
    +{.} {}                             {transition}
}

\BibSpec{book}{%
    +{}  {\PrintPrimary}                {transition}
    +{,} { \textit}                     {title}
    +{.} { }                            {part}
    +{:} { \textit}                     {subtitle}
    +{,} { \PrintEdition}               {edition}
    +{}  { \PrintEditorsB}              {editor}
    +{,} { \PrintTranslatorsC}          {translator}
    +{,} { \PrintContributions}         {contribution}
    +{,} { }                            {series}
    +{,} { \voltext}                    {volume}
    +{,} { }                            {publisher}
    +{,} { }                            {organization}
    +{,} { }                            {address}
    +{} { \PrintDate}                 {date}
    +{,} { }                            {status}
    +{}  { \parenthesize}               {language}
    +{}  { \PrintTranslation}           {translation}
    +{;} { \PrintReprint}               {reprint}
    +{.} { }                            {note}
    +{.} {}                             {transition}
}

\newcommand{\arxiv}[1]{\tt arxiv:\hspace{0pt}{\href{http://arxiv.org/abs/#1}{#1}}}

\renewcommand{\voltext}{\IfEmptyBibField{series}{\textbf}{\textbf}}

\catcode`\@=12

\numberwithin{equation}{section}\swapnumbers
\usepackage{parskip}

\usepackage[cmtip,matrix,arrow,curve]{xy}
\SelectTips{cm}{10}
\newcommand{\cxymatrix}[1]{\vcenter{\xymatrix@=15pt{#1}}}
\newcommand{\kxymatrix}[1]{\vcenter{\xymatrix@=5pt{#1}}}

\newcommand{\xysubseteq}{\ar@{}[r]|{\displaystyle\subseteq}}
\newcommand{\xysubseteqdown}{\ar@{}[d]|{\rotatebox{90}{$\supseteq$}}}

\usepackage{tikz}
\usetikzlibrary{calc}

\usepackage{aliascnt}

\newtheorem{theorem}{Theorem}[section]
\newaliascnt{lemma}{theorem}

\newtheorem{lemma}[lemma]{Lemma}
\aliascntresetthe{lemma}
\newaliascnt{corollary}{theorem}

\newtheorem{corollary}[corollary]{Corollary}
\aliascntresetthe{corollary}
\newaliascnt{proposition}{theorem}

\newtheorem{proposition}[proposition]{Proposition}
\aliascntresetthe{proposition}

\theoremstyle{definition}
\newaliascnt{definition}{theorem}

\newtheorem{definition}[definition]{Definition}
\aliascntresetthe{definition}

\newaliascnt{remark}{theorem}
\newtheorem{remark}[remark]{Remark}
\aliascntresetthe{remark}

\newtheorem*{remark*}{Remark}

\newaliascnt{example}{theorem}

\newtheorem*{example*}{Examples}

\aliascntresetthe{example}

\usepackage{enumitem}
\setlist[enumerate,1]{label=\textit{(\alph*)},ref=\textit{(\alph*}),noitemsep}
\setlist[enumerate,2]{label=\textit{\roman*)},ref=\textit{\roman*}),noitemsep}

\usepackage[capitalize]{cleveref}
\usepackage{microtype}

\renewcommand\[{\begin{equation}}
\renewcommand\]{\end{equation}}

\renewcommand\tilde{\widetilde}

\renewcommand\phi{\varphi}
\renewcommand\epsilon{\varepsilon}

\renewcommand\theta{\vartheta}
\renewcommand\rho{\varrho}

\newcommand\CC{{\mathbb C}}
\newcommand\FF{{\mathbb F}}

\newcommand\KK{{\mathbb K}}
\newcommand\NN{{\mathbb N}}

\newcommand\ZZ{{\mathbb Z}}
\newcommand\cA{{\mathcal A}}

\newcommand\cT{{\mathcal T}}

\newcommand\leer{\varnothing}
\newcommand\aq{{\overline{a}}}
\newcommand\bq{{\overline{b}}}
\newcommand\cq{{\overline{c}}}
\newcommand\dq{{\overline{d}}}
\newcommand\eq{{\overline{e}}}
\newcommand\Fq{{\overline{f}}}
\newcommand\gq{{\overline{g}}}

\newcommand\Rq{{\overline{R}}}

\newcommand\uq{{\overline{u}}}

\newcommand\wq{{\overline{w}}}

\newcommand\yq{{\overline{y}}}

\newcommand\into{\mathrel{\kern-3pt\xymatrix@=10pt{\ar@{>->}[r]&}\kern-5pt}}
\newcommand\inj{\ar@{>->}}
\newcommand\sur{\ar@{>>}}
\newcommand\auf{\twoheadrightarrow}
\newcommand\qur{\mathrel{\kern-10pt\raise2pt\hbox{\xymatrix@=10pt{\ar@{{ }>>}[r]&}\kern-5pt}}}
\newcommand\quot{\mathrel{\kern-5pt\raise2pt\hbox{\xymatrix@=10pt{&\ar@{{ }>>}[l]}\kern-10pt}}}
\newcommand\Auf[1]{\mathrel{\kern-3pt\xymatrix@=10pt{\ar@{>>}[r]^{#1}&}\kern-5pt}}
\newcommand\lrarrow{\mathop{\lower3pt\vbox{\baselineskip0pt\hbox{$\rightarrow$}\hbox{$\leftarrow$}}}}

\renewcommand\*{{\bf1}}
\newcommand\0{{\bf 0}}
\usepackage{ bbold }
\renewcommand\1{{\mathbb{1}}}

\newcommand\<{\langle}
\renewcommand\>{\rangle}
\def\hide#1{\hbox to 10pt{\hss$#1$\hss}}
\def\|#1|{\operatorname{#1}}
\newcommand\sO{{\mathrm s}}
\newcommand\qO{{\mathrm q}}

\makeatletter
\newcommand{\Times}{\gen@tens{\times}}
\newcommand{\gen@tens}[1]{%
  \@ifnextchar_{\gen@@tens{#1}}{\mathbin{#1}}%
}
\def\gen@@tens#1_#2{%
  \mathpalette\gen@@@tens{{#1}{#2}}%
}
\newcommand\gen@@@tens[2]{\mathbin{\gen@@@@tens#1#2}}
\newcommand\gen@@@@tens[3]{%
  \ifx#1\displaystyle
    \mathop{#2}\limits_{#3}%
  \else
    {#2}_{#3}%
  \fi
}
\makeatother
\let\Amalg=\amalg
\def\amalg{\mathbin{\scriptstyle\Amalg}}

\def\9#1/{{\color{red}#1}}

\newif\ifcharacterstart
\characterstarttrue

\title[]{The subobject decomposition in enveloping tensor\\categories}
\author[Friedrich Knop]{Friedrich Knop\hfill\ \\\small\rm Friedrich-Alexander-Universität Erlangen-Nürnberg}
\dedicatory{To the memory of Tonny A. Springer}


\begin{document}

\begin{abstract}
  To every regular category $\mathcal{A}$ equipped with a degree
  function $\delta$ one can attach a pseudo-abelian tensor category
  $\mathcal{T}(\mathcal{A},\delta)$. We show that the generating
  objects of $\mathcal{T}$ decompose canonically as a direct sum. In
  this paper we calculate morphisms, compositions of morphisms and
  tensor products of the summands. As a special case we
  recover the original construction of Deligne's category
  $\operatorname{\mathrm{Rep}} S_t$.
\end{abstract}
\maketitle         

\section{Introduction}

Deligne constructed in \cite{Deligne} the non-Tannakian tensor
category $\|Rep|S_t$ which can be interpreted as the category of
representations of the symmetric group on $t$ letters where $t$ does
not have to be a natural number. His construction is based on the fact
that the space of morphisms between certain objects of $\|Rep|S_n$,
$n\in\NN$, stabilize when $n$ goes to infinity.

In \cite{TERC}, another construction of $\|Rep|S_t$ was given which is
not based on this stabilization property. Instead, $\|Rep|S_t$ was
obtained as a twisted version of the category of relations attached to
$\cA=\mathsf{Set}^{\rm op}$, where $\mathsf{Set}$ denoted the category
of finite sets. This construction has the advantage that it readily
generalizes to much more general setting. In \cite{TERC} a
pseudo-abelian category tensor category $\cT(\cA,\delta)$ has been
constructed from any regular category $\cA$ which is equipped with a
degree function $\delta$. Moreover, precise conditions on $\cA$ and
$\delta$ were established for $\cT(\cA,\delta)$ to be a semisimple
(hence abelian) category. Taking for example for $\cA$ the category of
finite dimensional $\FF_q$-vector spaces, this leads to a category
$\|Rep|\|Gl|(V)$ where $V$ is an $\FF_q$-vector space with any number
$t\in\CC$ elements.

Both constructions, Deligne's and the authors, have in common that the
tensor category is built up from generating objects and a description
of the morphisms between them. The set of generating objects is not
the same, though. Even though both sets are parameterized by natural
numbers, Deligne's objects have much smaller morphism spaces between
them.

In this paper we elucidate how these two sets of generating objects
are related to each other. More generally, we define the analogues of
Deligne's generators in the more general setting of an arbitrary
regular category $\cA$ and show how $\cT(\cA,\delta)$ can be
constructed from them.

More specifically, each object $x$ of $\cA$ yields a generating object
$[x]$ of $\cT(\cA,\delta)$. The morphisms are linear combinations of
all relations. It turns out that $[x]$ contains a specific direct
summand $[x]^*$ and that $[x]$ is the direct sum of all $[y]^*$ where
$y$ runs through all subobjects of $x$.

We show that these $[x]^*$ are precisely Deligne's generators. For
this, we exhibit bases for the morphism spaces
$\|Hom|_\cT([x]^*,[y]^*)$ and calculate how composition is expressed
in terms of these bases. Moreover, we decompose the tensor product
$[x]^*\otimes[y]^*$ and calculate the tensor product of morphisms. The
formulas obtained turn out to specialize exactly to formulas used by
Deligne to define $\|Rep|S_t$. This shows conclusively that our
category $\cT(\mathsf{Set}^{\rm op},\delta)$ is equivalent to
$\|Rep|S_t$. More precisely, we show:

\begin{theorem}\label{thm:main}
  Let $\cA$ be a subobject finite, regular category, let $\delta$ be a
  degree function on $\cA$ and $\cT:=\cT(\cA,\delta)$.  For all
  objects $x$ and $y$ of $\cA$ let $R(x,y)$ be the set of subobjects
  of $x\times y$ such that both projections $r\to x$ and $r\to y$ are
  surjective. Then for every object $x$ of $\cA$ there is a direct
  summand $[x]^*$ of $[x]$ such that
  \[
    \bigoplus_{y\subseteq x}[y]^*\overset\sim\to[x]
  \]
  for all $x$. Moreover, the objects $[x]^*$ have the following
  properties:

  \begin{enumerate}

  \item\label{it:*def2} The objects $[x]^*$ are natural with respect
    to isomorphism and generate $\cT$ as a pseudo-abelian category.

  \item For all $x$ and $y$, the morphism space
    $\|Hom|_\cT([x]^*,[y]^*)$ has two natural (with respect to
    isomorphisms) bases $(r)$ and $\{r\}$ where $r$ runs through
    $R(x,y)$.

  \item The composition of morphisms is computed with formulas
    \eqref{eq:uuq} and \eqref{eq:sr2} below for the $(r)$- and the
    $\{r\}$-basis, respectively.

  \item Each tensor product $[x]^*\otimes[y]^*$ is naturally a direct
    sum of the objects $[r]^*$ where $r$ runs through $R(x,y)$. The
    unital, associativity, and symmetry constraints are induced by
    those of the direct product $x\times y$.

  \item\label{it:*def3} The tensor product of morphisms is computed
    with formulas \eqref{eq:tensorss} and \eqref{eq:tensorss2} below
    for the $(r)$- and the $\{r\}$-basis, respectively.

  \end{enumerate}
  
  In particular, $\cT$ is uniquely determined by the properties
  \ref{it:*def2} through \ref{it:*def3} (with either of the morphisms
  $(r)$ or $\{r\}$).
  
\end{theorem}
      
Summarizing: In the construction of $\cT(\cA,\delta)$ with the $[x]$
as basic objects, a basis of $\|Hom|_\cT([x],[y])$ is parameterized by
the set of all subobjects of $x\times y$. The formulas for composition
and tensor product of morphisms are very simple. In contrast, a basis
of the morphism space $\|Hom|_\cT([x]^*,[y]^*)$ for the $[x]^*$ is
parameterized by subobjects of $x\times y$ which map surjectively onto
each factor. There are much fewer of them. On the other hand the
formulas for composition and tensor product are much more involved.

Actually, the comparison of \cite{Deligne} and \cite{TERC} is a bit
more complicated. For $\|Rep|S_t$ the morphism space between $[x]$ and
$[y]$ is spanned by all partitions of the disjoint union $X\dot\cup
Y$. For $[x]^*$ and $[y]^*$ only those partitions are considered where
both $X$ and $Y$ meet every part in at most one element. It follows
that a partition of this kind corresponds to a \emph{gluing} of $X$
and $Y$, i.e., subsets $X_0\subseteq X$ and $Y_0\subseteq Y$ together
with a bijective map $X_0\overset\sim\to Y_0$. This alternate
description of the morphisms by way of gluings works for any so-called
Mal'tsev category. Now Deligne gives just a multiplication formula for
the $(r)$-morphisms in terms of gluings which we generalize in the
last section. Formula \eqref{eq:sr2} for the $\{r\}$-morphisms
seems to be new.

\begin{remark}
  For the benefit of the reader we indicate in detail how the
  construction in \cite{Deligne} is a special case of ours. The base
  category $\cA$ is chosen to be the opposite category
  $\mathsf{Set}^{\rm op}$ of the category of finite sets. Then we need
  first of all that $\cA$ is regular, exact, subobject finite, and, in
  the last section, Mal'cev.

  That $\mathsf{Set}^{\rm op}$ has these
  properties is most easily seen by observing that
  $\mathsf{Set}^{\rm op}$ is equivalent to the category
  $\mathsf{Bool}$ of finite Boolean algebras. In fact, the functor
  $C\mapsto\|Maps|(C,\FF_2)$ is an equivalence. Now $\mathsf{Bool}$ is
  the set of finite models of an algebraic theory (unital rings with
  all elements idempotent). Thus it is regular and exact by
  \cite{Borceux}*{Thm.~3.5.4}. It is clearly subobject finite. The
  Mal'cev property follows from the fact that the algebraic theory
  contains a group operation (namely addition, see
  \cite{Borceux2}*{Thm.~2.2.2, Ex.~2.2.5}).

  The degree function on $\cA$ is $\delta(C)=t^{|C|}$ where $t$ is a
  free variable and $|C|$ is the order of the finite set $C$. An
  epimorphism $e:x\auf x'$ in $\cA$ corresponds to an injective map
  $j:A'\hookrightarrow A$ of finite sets. Then our $\omega_e$ equals
  $P_A$ in the notation of \cite{Deligne}*{(2.10.2)}. This follows
  from \cite{TERC}*{Lemma\ 8.4 and 8.7 Example\ 1}
  
  The objects $[x]^*$ correspond to the objects $[U]$ in
  \cite{Deligne}*{2.12}. Our objects $[x]$ do not occur in
  \cite{Deligne} while, on the other side, Deligne's objects
  $\{\lambda\}$ will be considered in a forthcoming paper. Our
  morphisms $(r)$ and $\{r\}$ correspond to $(C)$ and $\{C\}$ in
  \cite{Deligne}*{2.12}, respectively.

  Finally the precise
  correspondence between the various Lemmas, Propositions and Theorems is (with \cite{Deligne} on top):

  \begin{center}
    \begin{tabular}{ccccccccccc}
  $\{r\}$&$(s)(r)$&$\{s\}\{r\}$&$\{v\}'\{u\}'$&$[x]^*\otimes[y]^*$&assoc./com.&$(r)\otimes(r')$&$\{r\}\otimes\{r'\}$
  \\
  \hline
  2.2&2.10&---&2.11&2.4&2.7&2.8&---
  \\
      \eqref{eq:rsrs}&\eqref{eq:uuq}&\eqref{eq:sr2}&\eqref{eq:Malcevprod}&\eqref{eq:tensorstar}&\eqref{eq:tensorass}/\eqref{eq:tensorcom}&\eqref{eq:tensorss}&\eqref{eq:tensorss2}
    \end{tabular}
\end{center}

Note that a more elaborate exposition of our construction in the case
$\cA=\mathsf{Set}^{\rm op}$ is contained in the paper \cite{CO} by
Comes and Ostrik.

\end{remark}

\section{The category \boldmath$\cT(\cA,\delta)$}

The monoidal category $\cT(\cA,\delta)$ is build from two ingredients,
a category $\cA$ and a degree function $\delta$.

The category $\cA$ has to be rich enough such that the usual
relational calculus makes sense. This means roughly that it has
images, pull-backs, and that images commute with pull-backs.

Let's be more precise. For any fixed object $x$ of $\cA$ the class of
monomorphisms $m:y\into x$ carries a transitive relation by
stipulating $m\le m'$ if $m$ factors through $m'$. Two monomorphisms
$m$, $m'$ are \emph{equivalent} if both $m\le m'$ and $m\ge
m'$. Equivalence classes of monomorphisms are called \emph{subobjects}
of $x$. The collection of all subobjects will be denoted by
$\sO(x)$. It is now partially ordered.

The \emph{image} $\|im|(f)$ of a morphism $f:x\to y$ is the minimal
subobject $m$ of $y$ through which $f$ factors (if one exists). The
image is also denoted by $f(x)$. If $f(x)=y$ then $f$ is by
definition an \emph{extremal epimorphism} (denoted $f:x\auf y$). More
generally, if the image of $f:x\to y$ is $m:u\into y$ then $f=me$
where $e:x\auf u$ is an extremal epimorphism. Moreover, this
epi-mono-factorization is unique upto unique isomorphism.

\begin{remark}
  The monomorphism/extremal epimorphism language is a bit too clumsy
  for our needs. Henceforth, these morphisms will be called injective
  and surjective, respectively. This is justified by the fact that in
  most examples injective/surjective morphisms are just
  injective/surjective maps. This applies e.g. to the categories of
  sets, groups or vector spaces with one notatble exception namely the
  category attached to Deligne's original category
  $\|Rep|(S_t)$. There $\cA=\mathsf{Set}^{\rm op}$ is the opposite
  category of the category of finite sets. Then injective morphisms
  correspond to surjective maps and vice verso. 
  Similarly, we write $u\subseteq v$ instead of
  $u\le v$ for subobjects $u,v\in\sO(x)$ and $u\cap v$ for the fiber
  product (intersection) $u\times_xv$.
\end{remark}

Now we formulate our main requirements on $\cA$.

\begin{definition}
  A category $\cA$ will be called \emph{regular} if the following
  conditions hold:

  \begin{itemize}[noitemsep]

  \item[\textbf{R0}] For every object the collection of its subobjects
    is a set.

  \item[\textbf{R1}] Every morphism has an image.

  \item[\textbf{R2}] There is a terminal object (denoted by $\*$).

  \item[\textbf{R3}] For every commutative diagram
    \[
      \cxymatrix{u\ar[r]\ar[d]&y\ar[d]\\x\ar[r]&z}
    \]
    the pullback $x\times_zy$ exists.

  \item[\textbf{R4}] The pull-back of a surjective morphism by an
    arbitrary morphism exists and is surjective.

  \end{itemize}

\end{definition}

\begin{remark}
  \emph{(1)} Condition \textbf{R4} can be rephrased as: Let $f:x\to z$
  and $g:y\to z$ be morphisms. Then
  \[
    f(x)\times_zy=f(x\times_zy)
  \]
  in the sense that one side exists if and only if the other does and
  in that case they are canonically isomorphic.

  \emph{(2)} The possibility for a pull-back not to exist is mainly
  included to accommodate the category of affine spaces over a field,
  since considering the empty set as an affine space creates more
  problems than it solves. For example the dimension formula does not
  hold for empty intersections. Nevertheless, one can form a new
  category $\cA^\leer$ by adjoining an initial object $\leer$ to $\cA$
  which is then finitely complete. More precisely, all finite limits
  which do not exist in $\cA$ exist in $\cA^\leer$ and are equal to
  $\leer$.

  \emph{(3)} The axioms of regular categories are usually formulated
  in terms of regular epimorphisms. It is not difficult to show that
  all surjective morphisms are regular (even effective). Thus the
  definitions are equivalent.
\end{remark}

As already mentioned, the definition of a regular category is tailored
to allow a calculus of relations. Recall that a \emph{relation}
between two objects $x$ and $y$ is a subobject $r$ of $x\times y$ or,
equivalently a jointly injective pair of morphisms $r\to x,y$. If $s$
is another relation between $y$ and $z$ then their \emph{product}
$r\circ s$ is the image of $p_x\times p_z:r\times_ys\to x\times z$
with the stipulation that $r\circ s=\leer$ if $r\times_ys$ does not
exist.

Axiom \textbf{R3} is now instrumental to show that the product of
relations is associative. This way one can define a new category
$\|Rel|\cA$ with the same objects as $\cA$ and with relations as
morphisms. The identity morphism in $\|Rel|\cA$ is given by the
\emph{diagonal} $\Delta x\subseteq x\times x$.

Given a relation $r$ between $x$ and $y$ then the \emph{adjoint
  relation} $r^\vee\subseteq y\times x$ is obtained by switching
factors. This way, $\|Rel|\cA$ becomes a symmetric monoidal category
with the tensor product being the Cartesian product $x\times y$ and
the terminal object $\*$ being the unital object. It is even
\emph{rigid} with every object being self-dual and
\[
  \cxymatrix{&x\ar[dl]\inj[dr]^\Delta\\\*&&\hbox
    to0pt{\hss$x\times
      x$\hss}} \quad\text{and}\quad
  \cxymatrix{&x\ar[dr]\inj[dl]_\Delta\\\hbox
    to0pt{\hss$x\times x$\hss}&&\*}
\]
being the evaluation and coevaluation morphism, respectively.

Our ultimate goal is to produce tensor categories, i.e. rigid
symmetric monoidal categories which are also abelian. It is only the
last property which is lacking for $\|Rel|\cA$. A first step in the
right direction is to make it additive. This is very easy by considering
linear combinations of relations as morphisms. But it turns out that
the ensuing category is quite degenerate. It turns out that a twist of
the construction makes things much better.

\begin{definition}
  Let $\cA$ be a regular category and $\KK$ any base ring. A
  \emph{degree function on $\cA$} is a map which assigns every
  surjective morphism $e:x\auf y$ an element $\delta(e)\in \KK$ with
  the following properties:
  \begin{itemize}[noitemsep]
  \item[{\bf D1}] $\delta(\|id|_x)=1$ for all $x$.
  \item[{\bf D2}] $\delta(\eq)=\delta(e)$ whenever $\eq$ is a
    pull-back of $e$.
  \item[{\bf D3}] $\delta(e\,\eq)=\delta(e)\,\delta(\eq)$ whenever $e$
    can be composed with $\eq$.
  \end{itemize}
\end{definition}

\begin{remark}
  \emph{(1)} If $\cA$ has no initial element we extend the degree
  function to $\cA^\leer$ by defining $\delta(\leer\to\leer):=1$. But
  beware: This does not define a degree function on $\cA^\leer$, since
  \textbf{D2} is violated for pull-backs along $\leer\to x$.

  \emph{(2)} It is sometimes convenient to extend $\delta$ to all
  morphisms $f:x\to y$ by defining
  $\delta(f):=\delta(x\auf\|im|f)$. Then \textbf{D2} holds
  unconditionally for all morphisms while \textbf{D3} is valid
  whenever $e$ is injective or $\eq$ is surjective.

  \emph{(2)} If $e$ and $\eq$ are isomorphic in an obvious sense then
  $\delta(\eq)=\delta(e)$. This is a consequence of \textbf{D2}. In
  particular $\delta(e)=1$ for all isomorphisms $e$.

  \emph{(3)} Every regular category $\cA$ has at least two degree
  function. First, the one with $\delta(e)=1$ for all $e$. Secondly,
  the one with $\delta(e)=0$ for all non-isomorphisms $e$.

  \emph{(4)} For most $\cA$ there aren't any other degree
  functions. This holds for example for the category $\mathsf{Set}$ of
  finite sets.

  \emph{(5)} As a positive example let $\cA$ be an abelian category
  such that every object $x$ has finite length $\ell(x)$. Let
  $t\in\KK$ be arbitrary. Then $\delta(e):=t^{\ell(\|ker|e)}$ is a
  non-trivial degree function.

  \emph{(6)} Another example is the category $\mathsf{Set}^{\rm
    op}$. Then $e:x\auf y$ is represented by an injective map $Y\to X$
  and $\delta(e)=t^{\#X-\#Y}$ is the degree function which gives rise
  to Deligne's category $\|Rep|S_t$.
\end{remark}

Now we can define an intermediate category $\cT^0(\cA,\delta)$.

\begin{definition} Let $\cA$ be a regular category, $\KK$ a field, and
  $\delta$ a $\KK$-valued degree function on $\cA$. Then the category
  $\cT^0=\cT^0(\cA,\delta)$ is defined as follows:

  \begin{enumerate}
  \item The \emph{objects} of $\cT^0$ are those of $\cA$. If an object
    $x$ of $\cA$ is regarded as an object of $\cT^0$ then we will
    denote it by $[x]$.

  \item A \emph{morphism} from $[x]$ to $[y]$ is a formal
    $\KK$-linear combination of relations between $x$ and $y$. The
    morphism corresponding to a relation $r$ is denoted by $\<r\>$. We
    set $\<\leer\>:=0$.

  \item The \emph{composition} of $\cT^0$-morphisms is defined as
    follows: let $r\into x\times y$ and $s\into y\times z$ be
    relations. Then
    \[\label{eq:compos1}
      \<s\>\<r\>:=
      \begin{cases}
        \delta(r\times_ys\auf s\circ r)\ \<s\circ r\>,&\text{if
          $r\times_ys$ exists;}\\
        0,&\text{otherwise.}
      \end{cases}
    \]

  \item The \emph{tensor product} of $[x]$ and $[y]$ is
    $[x]\otimes[y]=[x\times y]$. Let $\<r\>:[x]\to[x']$ and
    $\<s\>:[y]\to[y']$ be two morphisms given by relations $r$ and
    $s$. Then $\<r\>\otimes\<s\>:[x]\otimes[y]\to[x']\otimes[y']$ is
    given by the relation
    \[\label{eq:tensorrs}
      r\times s\into(x\times x')\times(y\times
      y')\overset\sim\to(x\times y)\times(x'\times y').
    \]
    The unit object is $\1:=[\*]$ (with $\*$ being the terminal object
    of $\cA$).

  \end{enumerate}
\end{definition}

Finally, we generate also direct sums and direct summands:

\begin{definition}

  The category $\cT(\cA,\delta)$ is the pseudo-abelian closure of
  $\cT^0(\cA,\delta)$, i.e., its objects are of the form $pX$ where
  $X=\bigoplus_\nu[x_\nu]$ is a formal finite direct sum of objects of
  $\cT^0$ with
  \[
    \|Hom|_\cT(X,Y):=\bigoplus_{\mu,\nu}\|Hom|_{\cT^0}([x_\mu],[y_\nu])
  \]
  and $p\in\|End|_\cT(X)$ is idempotent with morphisms
  \[
    \|Hom|_\cT(pX,qY):=q\|Hom|_\cT(X,Y)p.
  \]

\end{definition}

In \cite{TERC} it was shown that the category $\cT(\cA,\delta)$
inherits from $\cA$ the structure of a rigid symmetric monoidal
category with tensor product $[x]\otimes[y]=[x\times y]$ and dual
$(p[x])^\vee=p^\vee[x]$. In particular the associativity of the
product of morphisms is ensured by the axioms of a degree
function. Additionally, $\cT(\cA,\delta)$ is $\KK$-linear and
additive.

\section{An alternate construction of \boldmath$\cT(\cA,\delta)$}

Let $\cA$ be a regular category and $\delta$ a $\KK$-valued degree
function on $\cA$. In this section we give a description of
$\cT(\cA,\delta)$ or rather $\cT^0(\cA,\delta)$ in terms of generators
and relations while bypassing the category of relations.

Every morphism $f:x\to y$ gives rise to two morphisms in $\cT^0$ namely
\[
  [f]:=\<\Gamma_f\>:[x]\to[y] \text{ and
  }[f]^\vee=\<\Gamma_f^\vee\>:[y]\to[x]
\]
where $\Gamma_f\into x\times y$ is the graph of $f$. Moreover for
every relation
\[\label{eq:xarby}
  \cxymatrix{&r\ar[dl]_a\ar[dr]^b\\x&&y}
\]
we have
\[\label{eq:xry}
  \<r\>=[r\to y][r\to x]^\vee.
\]
So the objects $[x]$ and the morphisms $[f]$, $[f]^\vee$ generate
$\cT(\cA,\delta)$ as a pseudo-abelian category. It is easy to verify
that they satisfy the following relations:

\begin{itemize}[noitemsep]
\item[{\bf Rel1}]\label{eq:fgfg}
  $[f][g]=[fg]$ and $[g]^\vee[f]^\vee=[fg]^\vee$\\
  for all composable morphisms $f$, $g$,

\item[\textbf{Rel2}]\label{eq:baba}
  $[a']^\vee[b']=[b][a]^\vee$\\
  for each Cartesian diagram
  $\kxymatrix{&&r\ar[ld]_a\ar[rd]^b\\
    &x\ar[rd]_{b'}&&y\ar[ld]^{a'}&\\&&z&&}$
  
\item[\textbf{Rel3}]\label{eq:ff}
  $[f][f]^\vee=\delta(f)\|id|_{[y]}$\\
  for each surjective morphism $f:x\to y$.

\end{itemize}

Conversely:

\begin{theorem}\label{thm:con1}
  Let $\KK$ be a ring, let $\cA$ be a regular category, and let
  $\delta$ be a $\KK$-valued degree function on $\cA$. Then
  $\cT(\cA,\delta)$ is the free pseudo-abelian category which is
  generated by the objects $[x]$ with $x$ an object of $\cA$ and
  morphisms $[f]:[x]\to[y]$, $[f]^\vee:[y]\to[x]$ for each morphism
  $f:x\to y$ which are subject to the relations {\bf Rel1}, {\bf Rel2}
  and {\bf Rel3}.
\end{theorem}

\begin{proof}
  Let $\tilde\cT$ be the free category defined by the generators and
  relations above and let $F:\tilde\cT\to\cT(\cA,\delta)$ be the
  obvious functor. It suffices to show that $F$ is an isomorphism on
  $\|Hom|_{\tilde\cT}([x],[y])$ for all objects $x,y$ of $\cA$. Let $F(x,y)$
  be this map. Then $F(x,y)$ is surjective because of
  \eqref{eq:xry}. Thus, it suffices to show that
  $\|Hom|_{\tilde\cT}([x],[y])$ is linearly generated by all morphisms of
  the form \eqref{eq:xry}.

  For this, we first show that
  \[\label{eq:id}
    [\|id|_x]=[\|id|_x]^\vee=\|id|_{[x]}.
  \]
  It follows from {\bf Rel1} that both $[\|id|_x]$ and
  $[\|id|_x]^\vee$ are idempotents. Because of $\delta(\|id|_x)=1$ and
  {\bf Rel3} we have $[\|id|_x][\|id|_x]^\vee=\|id|_{[x]}$. Hence
  $[\|id|_x]=[\|id|_x][\|id|_x][\|id|_x]^\vee=[\|id|_x][\|id|_x]^\vee
  =\|id|_{[x]}$.

  The space of morphisms is spanned by products
  $\phi=\phi_1\ldots\phi_n$ where each $\phi_i$ is either of type
  $[f]$ or $[f]^\vee$. Because of {\bf Rel1} we may assume that no two
  adjacent morphisms are both of type $[f]$ or both of type
  $[f]^\vee$. Because of {\bf Rel2} we may also assume that there is
  no morphism of type $[f]^\vee$ followed by a morphism of type
  $[f]$. Thus, we have the possibilities $\phi=\|id|_{[x]}$ (case
  $n=0$), $\phi=[f]$, $\phi=[f]^\vee$, or $\phi=[g][f]^\vee$. Because
  of \eqref{eq:id}, we may assume that there are two morphisms
  $f:u\to x$, $g:u\to y$ such that $\phi=[g][f]^\vee$. Let $r$ be the
  image of $f\times g:u\to x\times y$ and $h:u\auf r$ the
  corresponding epimorphism. Consider the morphisms $\Fq:r\to x$ and
  $\gq:r\to y$. Then
  $[g][f]^\vee=[\gq h][\Fq
  h]^\vee=[\gq][h][h]^\vee[\Fq]^\vee=\delta(h)[\gq][\Fq]^\vee$. Thus
  $[g][f]^\vee$ is a multiple of a morphism of type \eqref{eq:xry}.
\end{proof}

\section{The subobject decomposition}

Let $\cA$ be a subobject finite regular category and $\delta$ a
$\KK$-valued degree function on $\cA$.

Of particular interest are the morphisms $[i]:[y]\to[x]$ where
$i:y\to x$ is injective. In this case $y\times_xy=y$ and therefore
\[\label{eq:i}
  [i]^\vee[i]=\|id|_{[y]}.
\]
It follows that $[i]$ is a split monomorphism and $[i]^\vee$ a split
epimorphism. Let $u$ be the image of $i$, i.e., the subobject
represented by $i$. Then \eqref{eq:i} implies that
$p_u:=[i][i]^\vee\in\|End|_\cT([x])$ is an idempotent for which one
easily checks that it depends only on $u$. Thus, every $u\in\sO(x)$
gives rise to a direct summand $p_u[x]$ of $[x]$ which is via $i$
isomorphic to $[y]$. Clearly $p_u=\<r\>$ where $r$ is the relation
\[
  \cxymatrix{&y\inj[dl]_i\inj[dr]^i\\x&&x}
\]
This implies that the collection of all idempotents $p_u$,
$u\in\sO(x)$, is linearly independent. One also easily checks that
\[
  p_up_v=p_{u\cap v}\text{ for all $u,v\in\sO(x)$}.
\]
In particular, the idempotents $p_u$ commute with each other. Thus,
they can be expressed by primitive idempotents $p_u^*$.

More precisely, let $\KK[\sO(x)]$ be the Möbius algebra of $\sO(x)$,
i.e., the $K$-vector space with basis $\sO(x)$ and product induced by
intersection. Then the formula
\[\label{eq:pp*}
  p_v=\sum_{u\subseteq y}p_u^*\text{ for all }v\in\sO(x).
\]
defines recursively a new basis $(p_y^*)_{y\in\sO(x)}$ satisfying
\[\label{eq:p*p*}
  p_u^*p_v^*=\delta_{u,v}p_u^*\quad\text{and}\quad
  p_u^*p_v=\begin{cases}p_u^*&\text{if }u\subseteq
    v\\0&\text{otherwise}\end{cases}.
\]
(see e.g. \cite{Stanley}*{Thm.~3.9.1}). Conversely, one has
\[\label{eq:p*p}
  p_v^*:=\sum_{u\subseteq v}\mu(u,v)p_u
\]
where $\mu(u,v)\in\ZZ$ is the Möbius function of the lattice
$\sO(x)$. Plugging in \eqref{eq:pp*} into \eqref{eq:p*p} and vice
verso one obtains the relations
\[\label{eq:summu}
  \sum_{u\subseteq v\subseteq w}\mu(v,w)=\delta_{u,w}=
  \sum_{u\subseteq v\subseteq w}\mu(u,v).
\]
Thus each subobject $u\subseteq x$ gives rise to a direct summand
$p_u^*[x]$ of $[x]$. We abbreviate $[x]^*:=p_x^*[x]$. Then
$p_u^*[x]=p_u^*p_u[x]\cong p_u^*[u]=[u]^*$. Thus we get the
\emph{subobject decomposition of $[x]$}
\[\label{eq:subobjdec}
  [x]=\sum_{u\subseteq x}p_u^*[x]=\sum_{u\subseteq x}[u]^*.
\]
Observe that because the $p_u$ are linearly independent, each summand
$[u]^*$ is non-zero.

Next, we describe the functorial properties of the subobject
decomposition.

\begin{lemma}Let $f:x\to y$ be a morphism and $z\subseteq y$. Then
  \[\label{eq:sur0}
    p_z[f]=[f]p_{f^{-1}(z)}.
  \]
\end{lemma}

\begin{proof}
  Let $\iota:z\into y$ be the inclusion. Then the two diagrams
  \[
    \cxymatrix{
      &&\hide{f^{-1}(z)}\inj[dl]^{\overline\iota}\ar[dr]_\Fq\\
      &x\ar@{=}[dl]\ar[dr]^f&&z\inj[dl]_\iota\inj[dr]_\iota\\
      x&&y&&y } \quad \cxymatrix{
      &&\hide{f^{-1}(z)}\ar@{=}[dl]\ar[dr]^{\overline\iota}\\
      &\hide{f^{-1}(z)}\ar[dl]^{\overline\iota}\ar[dr]_{\overline\iota}&&x\ar@{=}[dl]\ar[dr]^f\\
      x&&x&&y }
  \]
  represent the left and right hand side of \eqref{eq:sur0}. We
  conclude with $\iota\Fq=f\overline\iota$.
\end{proof}

From this we derive

\begin{lemma} Let $f:x\to y$ be a morphism and $u\subseteq x$,
  $z\subseteq y$ subobjects. Then
  \[\label{eq:sur3}
    p_z^*[f]p_u^*=\begin{cases} [f]p_u^*&\text{if
      }f(u)=z\\0&\text{otherwise.}
    \end{cases}
    \quad\text{and}\quad p_u^*[f]^\vee p_z^*=\begin{cases}
      p_u^*[f]^\vee&\text{if }f(u)=z\\0&\text{otherwise.}
    \end{cases}
  \]
  \[\label{eq:sur1}
    p_z^*[f]=\sum_{u\subseteq x\atop f(u)=z}[f]p_u^*
    \quad\text{and}\quad [f]^\vee p_z^*=\sum_{u\subseteq x\atop
      f(u)=z} p_u^*[f]^\vee.
  \]
\end{lemma}

\begin{proof}
  First, observe that the right hand formulas are just the adjoints of
  the ones on the left hand side.

  For $u\subseteq x$ multiply both sides of \eqref{eq:sur0} by
  $p_u^*$. Then
  \[\label{eq:pfp*}
    p_z[f]p_u^*\overset{\eqref{eq:p*p*}}=
    \begin{cases}[f]p_u^*&\text{if }f(u)\subseteq
      z\\0&\text{otherwise.}
    \end{cases}
  \]
  From this we get
  \[
    p_z^*[f]p_u^*\overset{\eqref{eq:p*p}}= \sum_{v\subseteq
      z}\mu(v,z)p_v[f]p_u^*\overset{\eqref{eq:pfp*}}=
    \sum_{f(u)\subseteq v\subseteq
      z}\mu(v,z)[f]p_u^*\overset{\eqref{eq:summu}}=
    \begin{cases}[f]p_u^*&\text{if }f(u)=z,\\0&\text{otherwise.}
    \end{cases}
  \]
  Then \eqref{eq:sur1} follows:
  \[
    p_z^*[f]\overset{\eqref{eq:pp*}}=\sum_{u\subseteq x}p_z^*[f]p_u^*
    \overset{\eqref{eq:sur3}}= \sum_{u\subseteq x\atop
      f(u)=z}[f]p_u^*.\qedhere
  \]
\end{proof}

\begin{corollary}
  Let $f:x\to y$ be a morphism and $u\subseteq x$, $v\subseteq y$
  subobjects. Then $[f]$ maps $[u]^*$ into $[f(u)]^*$ and $[f]^\vee$
  maps $[v]^*$ into $\bigoplus_{u\subseteq x\atop f(u)=v}[u]^*$ .
\end{corollary}

\begin{proof} By \eqref{eq:sur3} we have $(1-p_{f(u)}^*)[f]p_u^*=0$
  which means that $[u]^*$ is mapped into $[f(u)]^*$.
\end{proof}

For the next formula, we need a certain numerical invariant attached
to a surjective morphism. Since $\cA$ is subobject finite the set
$\sO(x)$ is a finite poset (even a lattice, possibly without a
minimum). Let $\mu(y,z)$ be its Möbius function and let $e:x\auf y$ be
surjective. Then we define
\[\label{eq:defomega}
  \omega_e:=\sum_{u\in\sO(x)\atop e(u)=y}\mu(u,x)\delta (u\auf y)\in
  \KK.
\]

\begin{lemma} Let $e:x\auf y$ be a surjective morphism. Then
  \[\label{eq:sur2}
    [e]\,p_x^*\,[e]^\vee=\omega_e\,p_y^*.
  \]
\end{lemma}

\begin{proof}
  Let $\iota:u\into x$ be a subobject. Then \eqref{eq:sur1} implies
  that $p_y^*[e]p_u=0$ unless $e(u)=y$. Using
  $p_u=[\iota][\iota]^\vee$ we get
  \[\begin{split}
      [e]\,p_x^*\,[e]^\vee&\overset{\eqref{eq:sur3}}=
      p_y^*\,[e]\,p_x^*\,[e]^\vee\overset{\eqref{eq:p*p}}=
      \sum_{u\subseteq x}
      \mu(u,x)\,p_y^*\,[e]\,p_u\,[e]^\vee=\\
      &=\sum_{u\subseteq x\atop
        e(u)=y}\mu(u,x)\,p_y^*\,[e\iota][e\iota]^\vee \overset{{\bf
          Rel3}}=\sum_{u\subseteq x\atop e(u)=y}\mu(u,x)\,\delta(u\auf
      y)\,p_y^*\overset{\eqref{eq:defomega}}=
      \omega_e\,p_y^*.\hfill\qedhere
    \end{split}
  \]
\end{proof}

\section{The subobject construction of \boldmath$\cT(\cA,\delta)$}

Let $\cA$ be a subobject finite regular category and $\delta$ a
$\KK$-valued degree function. In this section we show how to build up
$\cT$ from its generators $[x]^*$ thereby proving \cref{thm:main}.

We start with the description of morphisms:

\begin{definition} Let $R(x,y)\subseteq\sO(x\times y)$ denote the set
  of relations $r\subseteq x\times y$ such that both projections
  $a:r\to x$ and $b:r\to y$ are surjective. For every $r\in R(x,y)$
  one can define two morphisms $[x]^*\to[y]^*$ namely
  \[\label{eq:defrr}
    \begin{array}{lll}
      (r)&:=&p_y^*[b]\,p_r^*\,[a]^\vee p_x^*\overset{\eqref{eq:sur3}}=
              [b]\,p_r^*\,[a]^\vee\\
      \{r\}&:=&p_y^*\,\<r\>\, p_x^*=p_y^*\,[b]\,[a]^\vee\,
                p_x^*
    \end{array}
  \]
\end{definition}

These two sets of morphisms are related as follows:

\begin{lemma}
  \[\label{eq:rsrs}
    \{r\}=\sum_{s\in R(x,y)\atop s\subseteq r}(s) \quad\text{and}\quad
    (r)=\sum_{s\in R(x,y)\atop s\subseteq r}\mu(s,r)\{s\}.
  \]
  Here $\mu$ is the Möbius function of $\sO(x\times y)$.
\end{lemma}

\begin{proof}
  The first formula follows from
  \[
    \{r\}=p_y^*[b][a]^\vee p_x^ * \overset{\eqref{eq:pp*}}=
    \sum_{s\subseteq r}p_y^*[b]p_s^*[a]^\vee p_x^*
    \overset{\eqref{eq:sur3}}= \sum_{s\subseteq r\atop s\auf
      x,y}p_y^*[b]p_s^*[a]^\vee p_x^*= \sum_{s\in R(x,y)\atop
      s\subseteq r}(s).
  \]
  The second formula now follows from Möbius inversion.
\end{proof}

Now we have:

\begin{proposition}\label{prop:Basis}
  Each of the two sets of morphisms $(r)$ and $\{r\}$ with
  $r\in R(x,y)$ forms a basis of $\|Hom|_\cT([x]^*,[y]^*)$.
\end{proposition}

\begin{proof}
  Because of \eqref{eq:rsrs} it suffices to prove the assertion for
  $\{r\}$. Let $\Rq:=\sO(x\times y)$ the set of all relations,
  $R:=R(x,y)$, and $R':=\Rq\setminus R$ the complement. By definition
  we have
  \[
    \|Hom|([x]^*,[y]^*)=p_y^*\|Hom|([x],[y])p_x^*
    \subseteq\|Hom|([x],[y])=\KK[\<\Rq\>].
  \]
  It follows from \eqref{eq:sur1} that $p_y^*\<R'\>p_x^*=0$. Moreover,
  one checks easily that $p_y^*\<r\>p_x^*\in\<r\>+\KK[\<R'\>]$ for all
  $r\in R$. So
  $\<R\>\overset\sim\to \<\Rq\>/\<R'\>\overset\sim\to
  p_x^*\<\Rq\>p_x^*$ implies the assertion.
\end{proof}

Next we calculate the composition of basis elements. We do that for
the $(r)$-basis first.

\begin{lemma}\label{lem:produkt1}
  Let $r\in R(x,y)$ and $s\in R(y,z)$. Then
  \[\label{eq:uuq}
    (s)\,(r)=\sum_{u\in R(r,s)\atop u\subseteq r\times_y s}
    \omega_{u\auf\uq}\ (\uq).
  \]
  Here, $\uq$ is the image of $u$ in $x\times z$.
\end{lemma}

\begin{proof}
  With $t:=r\times_y s$ we get the diagram
  \[
    \cxymatrix{&&t\sur[ld]_{a''}\sur[rd]^{b''}\\
      &r\sur[ld]_a\sur[rd]^{b}&&s\sur[ld]_{a'}\sur[rd]^{b'}&\\x&&y&&z}
  \]
  Now we calculate
  \[
    \begin{split}
      (s)(r)=&[{b'}]\,p_s^*\ [a']^\vee\,[b]\ p_r^*\,[a]^\vee
      \overset{{\bf Rel2}}=[{b'}]\,p_s^*\
      [b'']\,[a'']^\vee\ p_r^*\,[a]^\vee =\\
      =&\sum_{u\subseteq t}[{b'}]\,p_s^*\ [b'']\,p_u^*\,[a'']^\vee\
      p_r^*\,[a]^\vee\overset{\eqref{eq:sur1}}= \sum_{u\in R(r,s)\atop
        u\subseteq
        t}[{b'}][b'']\,p_u^*\,[a'']^\vee[a]^\vee\overset{\eqref{eq:uuxz}}=\\
      =&\sum_{u\in R(r,s)\atop u\subseteq
        t}[\bq][f]\,p_u^*\,[f]^\vee[\aq]^\vee
      \overset{\eqref{eq:sur2}}=\sum_{u\in R(r,s)\atop u\subseteq
        t}\omega_f\,[\bq]p_\uq^*[\aq]^\vee =\sum_{u\in R(r,s)\atop
        u\subseteq t}\omega_f\ (\uq)
    \end{split}
  \]
  with $\uq$ the image of $u$ in $x\times z$:
  \[\label{eq:uuxz}
    \xymatrix@=15pt{&&u\sur[d]_f\sur[ddll]_{aa''}\sur[ddrr]^{{b'}b''}\\
      &&\uq\sur[dll]^\aq\sur[drr]_\bq\\x&&&&z}\hfill
  \]
\end{proof}

The same for $\{r\}$:

\begin{lemma}\label{lem:produkt2}
  Let $r\in R(x,y)$ and $s\in R(y,z)$. For a subobject $y'\subseteq y$
  put, by abuse of notation,
  \[
    r\times_{y'}s:=r\times_yy'\times_ys \quad\text{and}\quad
    r\circ_{y'}s:=\|im|(r\times_{y'}s\to x\times z).
  \]
  Then
  \[\label{eq:sr2}
    \{s\}\,\{r\}=
    \sum_{\substack{y'\subseteq y\\
        r\times_yy'\auf x\\
        y'\times_ys\auf z}}
    \mu(y',y)\delta(r\times_{y'}s\to
    r\circ_{y'}s)\{r\circ_{y'}s\}.
  \]
\end{lemma}

\begin{proof}
  By definition
  \[
    \{s\}\{r\}=p_z^*[{b'}][a']^\vee\ p_y^*\ [b][a]^\vee p_x^*=
    \sum_{y'\subseteq y}\mu(y',y)p_z^*[{b'}][a']^\vee\ p_{y'}\
    [b][a]^\vee p_x^*
  \]
  Taking preimages of $y'$ yields the diagram
  \[
    \cxymatrix{
      &&r\times_{y'}s\sur[dl]\sur[dr]\ar@{..>>}[d]^f\\
      &r\times_y{y'}\sur[dr]\ar[dl]&r\circ_{y'}s\ar@{..>}[dll]^\aq\ar@{..>}[drr]_\bq&{y'}\times_ys\sur[dl]\ar[dr]\\
      x&&y'&&z\\
    }
  \]
  Thus
  \[
    p_z^*[{b'}][a']^\vee\ p_{y'}\ [b][a]^\vee p_x^*=p_z^*[\bq]\
    [f][f]^\vee\ [\aq]^\vee p_x^*=\delta(f)\ p_z^*[\bq][\aq]^\vee
    p_x^*.
  \]
  The last expression equals $\delta(f)\,\{r\circ_{y'}s\}$ if both
  $\aq$ and $\bq$ are surjective and is zero otherwise.
\end{proof}

\begin{remark}
  Using equations \eqref{eq:rsrs} and \eqref{eq:uuq} one can express
  $\{s\}\{r\}$ also in terms of the $(r)$-basis resulting in
  \[\label{eq:sr3}
    \{s\}\,\{r\}=\sum_{u\subseteq r\times_y s\atop u\auf
      x,y,z}\omega_{u\auf\uq}\ (\uq)
  \]
  where $\uq$ denotes the image of $u$ in $x\times z$. For exact
  Mal'tsev categories we will show the more stringent formula
  \eqref{eq:Malcevprod} below.
\end{remark}

\cref{prop:Basis} and \cref{lem:produkt1} or \cref{lem:produkt2}
describe $\cT(\cA,\delta)$ as an additive category. Now we turn to its
monoidal structure.

\begin{lemma}
  Let $x$ and $y$ be objects of $\cA$. Then
  \[\label{eq:potimesp}
    p_x^*\otimes p_y^*=\sum_{r\in R(x,y)}p_r^*\ \in\ \|End|([x\times
    y])
  \]
  In particular, there is a canonical isomorphism
  \[\label{eq:tensorstar}
    [x]^*\otimes[y]^*\cong\bigoplus_{r\in R(x,y)}[r]^*
  \]
\end{lemma}

\begin{proof}
  The idempotent $P:=p_x^*\otimes p_y^*$ is contained in the algebra
  spanned by the projections $p_{u\times v}$ with $u\subseteq x$ and
  $v\subseteq y$. So it must be a sum $P=\sum_{r\in R_0}p_r^*$ of
  minimal idempotents where $R_0$ is a certain subset of
  $\sO(x\times y)$. Let $u\subseteq x$ and $v\subseteq y$ with
  $u\ne x$ or $v\ne y$. Then $p_u\otimes p_v=p_{u\times v}$ implies
  \[
    0\overset{\eqref{eq:pp*}}=p_u p_x^*\otimes p_vp_y^*=\sum_{r\in
      R_0}p_{u\times y}p_r^*=\sum_{r\in R_0\atop r\subseteq u\times
      v}p_r^*.
  \]
  So the right hand sum is empty which means that
  $R_0\subseteq R(x,y)$. Because $[x]^*$ is self-dual and because of
  $\dim_\KK\|Hom|_\cT(\1,[x]^*)=1$ we have
  \[
    \begin{split}
      |R(x,y)|=&\dim_\KK\|Hom|([x]^*,[y]^*)=\dim_\KK\|Hom|(\1,[x]^*\otimes[y]^*)=\\
      =&\sum_{r\in R_0}\dim_\KK\|Hom|_\cT(\1,[r]^*)=|R_0|.
    \end{split}
  \]
  Thus $R_0=R(x,y)$.
\end{proof}

\begin{remark}\label{eq:tensorstar2}
  Formula \eqref{eq:tensorstar} generalizes readily to tensor products
  with more than two factors. More precisely, let $x_1,\ldots,x_n$ be
  objects of $\cA$. Then
  \[
    [x_1]^*\otimes\ldots\otimes[x_n]^*\cong\bigoplus_r\ [r]^*
  \]
  where $r$ runs through all subobjects of $x_1\times\ldots\times x_n$
  such that the projections $r\to x_i$ are surjective for
  $i=1,\ldots,n$. In particular, the associativity morphism
  $(x\times y)\times z\overset\sim\to x\times(y\times z)$ yields the
  associativity constraint
  \[\label{eq:tensorass}
    ([x]^*\otimes[y]^*)\otimes[z]^*\overset\sim\to
    [x]^*\otimes([y]^*\otimes[z]^*).
  \]
  Likewise, the canonical isomorphism
  $x\times y\overset\sim\to y\times x$ yields the commutativity
  constraint
  \[\label{eq:tensorcom}
    [x]^*\otimes[y]^*\overset\sim\to[y]^*\otimes[x]^*.
  \]
  
\end{remark}

Next we investigate the tensor product as a functor. For this we
need to determine for any morphisms $\phi:[x]^*\to[y]^*$ and
$\psi:[x']^*\to[y']^*$ the matrix coefficients of
\[
  \bigoplus_{u\in R(x,x')}[u]^*=[x]^*\otimes[x']^*
  \overset{\phi\otimes\psi}\longrightarrow
  [y]^*\otimes[y']^*=\bigoplus_{v\in R(y,y')}[v]^*.
\]
Clearly, it suffices to do this for a basis and we start with the
$(r)$-basis.

\begin{lemma} Let
  \[
    \cxymatrix{&r\sur[ld]_a\sur[rd]^b\\x&&y}\text{ and }
    \cxymatrix{&r'\sur[ld]_{a'}\sur[rd]^{b'}\\x'&&y'}
  \]
  be elements of $R(x,y)$ and $R(x',y')$, respectively. For each
  $w\in R(r,r')$ define
  \[
    r_w:=(a\times a')(w)\subseteq x\times x', \qquad r'_w:=(b\times
    b')(w)\subseteq y\times y'
  \]
  and the morphism
  \[
    \tau_w: [x]^*\otimes[x']^*\auf[r_w]^*\overset{(w)}
    \longrightarrow[r'_w]^*\into[y]^*\otimes[y']^*
  \]
  where $(w)$ is induced by the inclusion $w\into r_w\times
  r_w'$. Then
  \[\label{eq:tensorss}
    (r)\otimes(r')=\sum_{w\in R(r,r')}\tau_w.
  \]
\end{lemma}

\begin{proof}
  The morphism $w\to r_w\times r_w'$ is in fact injective since
  $w\into r\times r'\into x\times y\times x'\times y'$ and therefore
  also $w\to r_w\times r_{w'}\to x\times x'\times y\times y'$ is
  injective. Below is a diagram of the setup. All four quadrilaterals
  commute.
  \[ \vcenter{\xymatrix@=10pt{
        &&&w\sur[dl]\sur[dr]\ar@{.>>}[dll]\ar@{.>>}[drr]\\
        &r_w\ar@{.>>}[ddl]\ar@{.>>}[ddr]&r\sur[ddll]\sur[ddrr]&&r'\sur[ddll]\sur[ddrr]&r_w'\ar@{.>>}[ddl]\ar@{.>>}[ddr]\\
        &&&&&&\\
        x&&x'&&y&&y' }}
  \]
  Let $\aq:=a\times a'$ and $\bq:=b\times b'$. Then
  \[
    \begin{split}
      (r)\otimes(r')=&[b]\,p_r^*\,[a]^\vee\otimes[b']\,p_{r'}^*\,[a']^\vee=
      [b\times b'](p_r^*\otimes p_{r'}^*)[a\times a']^\vee
      \overset{\eqref{eq:potimesp}}=\\
      =&\sum_{w\in
        R(r,r')}[\bq]\,p_w^*\,[\aq]^\vee\overset{\eqref{eq:sur1}}=
      \sum_{w\in
        R(r,r')}p_{r_w'}^*\,[\bq]\,p_w^*\,[\aq]^\vee\,p_{r_w}^*=
      \sum_{w\in R(r,r')}\tau_w.\qedhere
    \end{split}
  \]
\end{proof}

Now the same for the $\{r\}$-basis.

\begin{lemma} Let
  \[
    \cxymatrix{&r\sur[ld]_a\sur[rd]^b\\x&&y}\text{ and }
    \cxymatrix{&r'\sur[ld]_{a'}\sur[rd]^{b'}\\x'&&y'}
  \]
  be elements of $R(x,y)$ and $R(x',y')$, respectively. For each
  $u\in R(x,x')$ and $v\in R(y,y')$ define
  \[
    w_{u,v}:=u\Times_{x\times x'}(r\times r')\Times_{y\times y'}v
  \]
  (see diagram \eqref{eq:wuv} below) and the morphism
  \[
    \sigma_{u,v}: [x]^*\otimes[x']^*\auf[u]^*\overset{\{w_{u,v}\}}
    \longrightarrow[v]^*\into[y]^*\otimes[y']^*
  \]
  Then
  \[\label{eq:tensorss2}
    \{r\}\otimes\{r'\}=\sum_{u\in R(x,x')\atop v\in
      R(y,y')}\sigma_{u,v}.
  \]
\end{lemma}

\begin{proof} First we claim that $w_{u,v}\to u\times v$ is
  injective. For this, let $\wq$ be the image of $w_{u,v}$ in
  $u\times v$. Then $w_{u,v}\to x\times x'\times y\times y'$ would
  factor through $\wq$. Then the same holds for
  $w_{u,v}\to r\times r'$ since
  $r\times r'\to x\times y\times x'\times y'$ is injective. So also
  the injective morphism $w_{u,v}\to u\times r\times r'\times v$
  factors through $\wq$ which shows the claim $w_{u,v}=\wq$.
  \[\label{eq:wuv}
    \vcenter{\xymatrix@=10pt{
        &&&w_{u,v}\ar@{.>>}[dl]\ar@{.>>}[dr]\ar@{.>>}[dll]\ar@{.>>}[drr]\\
        &u\sur[ddl]\sur[ddr]&r\sur[ddll]\sur[ddrr]&
        &r'\sur[ddll]\sur[ddrr]&v\sur[ddl]\sur[ddr]\\
        &&&&&&\\
        x&&x'&&y&&y' }}
  \]
  With $\aq:=a\times a'$ and $\bq:=b\times b'$ we get
  \[
    \{r\}\otimes\{r'\}=(p_y^*\otimes
    p_{y'}^*)[\bq][\aq]^\vee(p_x^*\otimes p_{x'}^*)
    \overset{\eqref{eq:potimesp}}= \sum_{u\in R(x,x')\atop v\in
      R(y,y')}p_v^*[\bq][\aq]^\vee p_u^*= \sum_{u\in R(x,x')\atop v\in
      R(y,y')}\sigma_{u,v}.\qedhere\qed
  \]
\end{proof}

\section{A formula for exact Mal'tsev categories}

As explained in the introduction, formula \eqref{eq:sr2} for the
product $\{s\}\{r\}$ does not appear in \cite{Deligne}. Instead
another identity is proven (\cite{Deligne}*{2.11}) which we consider
now in our framework. For that we have to assume that $\cA$ is exact
and Mal'tsev.

\begin{definition}
  A regular category $\cA$ is \emph{Mal'tsev} if for every object $x$
  any subobject $r\subseteq x\times x$ containing the diagonal is an
  equivalence relation. The category is \emph{exact} if for every
  equivalence relation $r\subseteq x\times x$ there is a surjective
  morphism $x\auf y$ with $r=x\times_yx$.
\end{definition}

\begin{example*}

\end{example*}

The quotient object $y$ of $x$ is uniquely determined by
$r$. Therefore, in an exact Mal'tsev category there is a duality between
subobjects of $x\times x$ containing the diagonal and quotient objects
of $x$. This can be generalized to relations between different objects
$x$ and $y$. For this we define $R_\qO(x,y)$ to be the set of
isomorphism classes of diagrams
    \[\label{eq:corelation}
      \kxymatrix{x\sur[dr]_a&&y\sur[dl]^b\\&u}.
    \]

    It is clear that for every $u\in R_\qO(x,y)$ the fiber product
    $x\times_uy$ is in $R(x,y)$.

\begin{lemma}[{\cite{CKP}*{Thm.~5.5}}]
  Let $x$ and $y$ be objects of an exact Mal'tsev category $\cA$. Then the
  map
  \[\label{lem:push-pull}
    R_\qO(x,y)\to R(x,y):u\mapsto r=x\times_uy
  \]
  is bijective, the inverse being the push-out
  $r\mapsto u=x\amalg_ry$.
\end{lemma}

Now for every $u\in R_\qO(x,y)$ we define
\[\label{eq:varr}
  \{u\}' :=p_y^*[b]^\vee[a]p_x^*\overset{\mathbf{Rel2}}=\{x\times_uy\}.
\]
There is an order relation on $R_\qO(x,y)$ by defining $u\le v$ when
$x,y\auf v$ factors through $x,y\auf u$. This way, $x,y\auf\*$ becomes
the maximal element of $R_\qO(x,y)$. Moreover, \eqref{lem:push-pull}
is order preserving. In particular, for any $u\in R_\qO(x,y)$ the
interval $[u,\*]$ can be identified with the set $\qO(u)$ of quotient
objects and therefore forms a lattice. Therefore it carries a Möbius
function.

and \eqref{eq:rsrs} becomes
\[\label{eq:rrss}
  \{u\}'=\sum_{t\in R_\qO(x,y)\atop t\le u}(x\times_ty)
  \quad\text{and}\quad (r)=\sum_{u\in R_\qO(x,y)\atop u\le
    x\amalg_ry}\mu_\qO(u,x\amalg_ry)\{u\}'.
\]

  
Then the following multiplication formula generalizes
\cite{Deligne}*{2.11}:

\begin{lemma}
  Let $\cA$ be an exact Mal'tsev category. Let $u\in R_\qO(x,y)$,
  $v\in R_q(y,z)$. Let $\yq:=\|im|(y\to u\times v)$ and
  $w:=u\amalg_yv$ (such that $u,v,\yq,w$ forms a pull-back diagram,
  see \eqref{eq:produkt1} below). Then
  \[\label{eq:Malcevprod}
    \{v\}'\,\{u\}'=\omega_{y\auf\yq}\sum_{t\in R_\qO(u,v)\atop t\le w}
    \mu_\qO(t,w)\ \{t\}'.
  \]
\end{lemma}

\[\label{eq:produkt1}
  \vcenter{\xymatrix@R=1pc@C=2pc{
      &x\sur[ddr]_a&&y\sur[ddl]_b\sur[d]^f\sur[ddr]^c&&z\sur[ddl]^d\\
      &&&\sur@<0.2mm>[dl]\sur@<-0.2mm>[dl]\yq\sur[dl]^\bq\sur@<0.2mm>[dr]\sur@<-0.2mm>[dr]\sur[dr]_\cq\\
      &&u\sur[dr]^\aq\sur@<0.2mm>[ddr]\sur@<-0.2mm>[ddr]\sur[ddr]_{a'}&&v\sur[dl]_\dq\sur@<0.2mm>[ddl]\sur@<-0.2mm>[ddl]\sur[ddl]^{d'}\\
      &&&t\sur[d]&&&\\
      &&&w }}
\]

\begin{proof}
Below we write ${}_v\{t\}'_u$ or ${}_z\{t\}'_x$ to indicate whether
we consider $\{t\}'$ as a morphism $[u]^*\to[v]^*$ or $[x]^*\to[z]^*$. Then we have:
\[
    \begin{split}
      \{v\}'\{u\}' ={}&p_z^*[d]^\vee[c]\ \ p_y^*\ \ [b]^\vee[a]p_x^*=\\
      ={}&p_z^*[d]^\vee\ \ [\cq][f]p_\yq^*[f]^\vee[\bq]^\vee\ \ [a]p_x^*\overset{\eqref{eq:sur2}}=\\
      ={}&\omega_f\ p_z^*[d]^\vee\ [\cq]p_\yq^*[\bq]^\vee\ [a]p_x^*\overset{\eqref{eq:defrr}}=\\
      ={}&\omega_f\ p_z^*[d]^\vee\ (\yq)
      \ [a]p_x^*\overset{\eqref{eq:rrss}}=\\
      ={}&\omega_f\ p_z^*[d]^\vee\left( \sum_{t\in R_\qO(u,v)\atop
          t\le w}
        \mu_\qO(t,w)\ {}_v\{t\}_u' \right) [a]p_x^*\overset{\eqref{eq:varr}}=\\
      ={}&\omega_f\sum_{t\in R_\qO(u,v)\atop t\le w}\mu_\qO(t,w)\
      p_z^*[d]^\vee\ \
      p_v^*[\dq]^\vee[\aq]p_u^*\ \ [a]p_x^*\overset{\eqref{eq:sur1}}=\\
      ={}&\omega_f\sum_{t\in R_\qO(u,v)\atop t\le w}\mu_\qO(t,w)\
      p_z^* [\dq
      d]^\vee[\aq a]p_x^*\overset{\eqref{eq:varr}}=\\
      ={}&\omega_f\sum_{t\in R_\qO(u,v)\atop t\le w}\mu_\qO(t,w)\
      {}_z\{t\}_x'\qedhere
    \end{split}
  \]
\end{proof}


\end{document}